%% file: ArticoloConMichelaCerquaSpringer.tex
\documentclass[graybox, envcountsame]{svmult}

\usepackage{type1cm}        
\usepackage{makeidx}         
\usepackage{graphicx}        
\usepackage{multicol}        
\usepackage[bottom]{footmisc}

\usepackage{newtxtext}       %
\usepackage{newtxmath}       
\usepackage{tikz}
\usepackage{tikz-cd}
\usepackage{latexsym}
\usepackage[all]{xy}
\newcommand{\Cal}[1]{{\mathcal #1}}
\newtheorem{examples}{Examples}

\def\R{ \mathbb{R} }
\def\C{ \mathbb{C} }

\newcommand{\End}{\operatorname{End}}

\newcommand{\Hom}{\operatorname{Hom}}

\newcommand{\Der}{\operatorname{Der}}
\newcommand{\PreDer}{\operatorname{PreDer}}

\def\cleft{\hbox{[\kern-.16em\hbox{[}}}
\def\cright{\hbox{]\kern-.16em\hbox{]}}}

\newcommand{\PreL}{\mathsf{PreL}}
\newcommand{\AntiPreL}{\mathsf{AntiPreL}}

\DeclareMathOperator{\op}{op}
\DeclareMathOperator{\ad}{ad}
\DeclareMathOperator{\tr}{tr}
\DeclareMathOperator{\Spec}{Spec}

\makeindex             

\begin{document}

\title*{Pre-Lie algebras, their multiplicative lattice, and idempotent endomorphisms}
\author{Michela Cerqua and Alberto Facchini}
\institute{Michela Cerqua \email{michela.cerqua@studenti.unipd.it}
\and Alberto Facchini \at Dipartimento di Matematica "Tullio Levi Civita", Universit\`a di Padova, 35121 Padova, Italy
\email{facchini@math.unipd.it}}
%
%
\maketitle

\abstract*{We introduce the notions of pre-morphism and pre-derivation for arbitrary non-associative algebras over a commutative ring $k$ with identity. These notions are applied to the study of pre-Lie $k$-algebras and, more generally, Lie-admissible $k$-algebras. Associating with any algebra $(A,\cdot)$ its sub-adjacent anticommutative algebra $(A,[-,-])$ is a functor  from the category of $k$-algebras with pre-morphisms to the category of anticommutative $k$-algebras. We describe the commutator of two ideals of a pre-Lie algebra, showing that the condition (Huq=Smith) holds for pre-Lie algebras. This allows to make use of all the notions concerning multiplicative lattices in the study of the multiplicative lattice of ideals of a pre-Lie algebra. We study idempotent endomorphisms of a pre-Lie algebra $L$, i.e., semidirect-product decompositions of $L$ and bimodules over $L$.}

\abstract{We introduce the notions of pre-morphism and pre-derivation for arbitrary non-associative algebras over a commutative ring $k$ with identity. These notions are applied to the study of pre-Lie $k$-algebras and, more generally, Lie-admissible $k$-algebras. Associating with any algebra $(A,\cdot)$ its sub-adjacent anticommutative algebra $(A,[-,-])$ is a functor  from the category of $k$-algebras with pre-morphisms to the category of anticommutative $k$-algebras. We describe the commutator of two ideals of a pre-Lie algebra, showing that the condition (Huq=Smith) holds for pre-Lie algebras. This allows to make use of all the notions concerning multiplicative lattices in the study of the multiplicative lattice of ideals of a pre-Lie algebra. We study idempotent endomorphisms of a pre-Lie algebra $L$, i.e., semidirect-product decompositions of $L$ and bimodules over $L$.}

\section*{Introduction}

The aim of this paper is to present pre-Lie algebras from the point of view of their multiplicative lattice of ideals, and to study their idempotent endomorphisms. Pre-Lie algebras were first introduced and studied in \cite{55} by Vinberg. He applied them to the study of convex homogenous
cones. He called ``left-symmetric algebras'' the algebras we call pre-Lie algebras in this paper. 

We present a notion of pre-morphism and pre-derivation for arbitrary non-associative algebras over a commutative ring $k$ with identity, and apply it to the study of pre-Lie $k$-algebras and, more generally, Lie-admissible $k$-algebras. Associating with any pre-Lie algebra $(A,\cdot)$ its sub-adjacent  Lie algebra $(A,[-,-])$ is a functor from the category $\PreL_{k,p}$ of pre-Lie $k$-algebras with pre-morphisms to the category of Lie $k$-algebras. 
We introduce the notion of module $M$ over a pre-Lie algebra $L$ and, like in the case of associative algebras, it is possible to do it in two equivalent ways, via a suitable scalar multiplication $L\times M\to M$ or as a $k$-module $M$ with a pre-morphism $\lambda\colon (L,\cdot)\to (\End(_kM),\circ)$. The category of modules over a pre-Lie $k$-algebra $(L,\cdot)$ is isomorphic to  the category of modules over its sub-adjacent  Lie $k$-algebra $(L,[-,-])$. We then consider the commutator of two ideals in a pre-Lie algebra. In particular we show that the condition (Huq=Smith) holds for pre-Lie algebras. With the notion of commutator at our disposal, the lattice of ideals of a pre-Lie algebra becomes a multiplicative lattice \cite{F, FFJ}. As a consequence we immediately get the notions of abelian pre-Lie algebra, prime ideal, prime spectrum of a pre-Lie algebra, solvable and nilpotent pre-Lie algebras, metabelian and hyperabelian pre-Lie algebras, centralizer, and center.

We then consider idempotent endomorphisms of a pre-Lie algebra, because they immediately show what semi-direct products of pre-Lie algebras are, what the action of a pre-Lie algebra on another pre-Lie  algebra is, and lead us to the notion of bimodule over a pre-Lie algebra. We study the ``Dorroh extensions'' of pre-Lie algebras. Like in the associative case, we get a category equivalence between the category $\PreL_{k}$ and the category of pre-Lie algebras with identity and with an augmentation. 

	\section{Preliminary notions on non-associative $k$-algebras}\label{1}
	
	Let $k$ be a commutative ring with identity. In this article, a {\em $k$-algebra}  is a $k$-module $_kM$ with a further $k$-bilinear operation $M\times M\rightarrow M$, $(x,y)\mapsto xy$ (equivalently, a $k$-module morphism $M\otimes_kM\rightarrow M$). A {\em subalgebra} (an {\em ideal}, resp.) of $M$ is a $k$-submodule $N$ of $M$ such that $xy\in N$ for every $x,y\in N$ ($xn\in N$	and $nx\in N$ for every $x\in M$ and $n\in N$, resp.)	As usual, if $N$ is an ideal of $M$, the quotient $k$-module $M/N$ inherits a $k$-algebra structure. There is a one-to-one correspondence between the set of all ideals $N$ of $M$ and the set of all congruences on $M$, that is, all equivalence relations $\sim$ on $M$ for which $x\sim y$ and $z\sim w$ imply $x+z\sim y+w$, $\lambda x\sim\lambda y$ and $xz\sim yw$ for every $x,y,z,w\in M$ and every $\lambda\in k$. The {\em opposite} $M^{{\op}}$ of an algebra $M$ is defined taking as multiplication in $M^{\op}$ the mapping $(x,y)\mapsto yx$.

		If $M $ and $M '$ are two $k$-algebras, a $k$-linear mapping $\varphi\colon M \rightarrow M '$ is a {\em $k$-algebra homomorphism} if $\varphi(xy)=\varphi(x)\varphi(y)$ for every $x,y\in M $. Clearly, $k$-algebras form a variety in
the sense of Universal Algebra. Moreover, it is a variety of $\Omega$-groups, 
that is, a variety which is pointed (i.e., it has exactly one constant) and
has amongst its operations and identities those of the variety of groups.  It follows that $k$-algebras form a semiabelian category. Other examples of $\Omega$-groups are abelian groups, non-unital rings, 
commutative algebras, modules and Lie algebras.

If $M$ is any $k$-algebra, its endomorphisms form a monoid, that is, a semigroup with a two-sided identity, with respect to composition of mappings $\circ$. A {\em derivation} of a $k$-algebra $M$ is any $k$-linear mapping $D\colon M\to M$ such that $D(xy)=(D(x))y+x(D(y))$ for every $x,y\in M$. For any $k$-algebra $M$, we can construct the {\em $k$-algebra of derivations} $\Der_k(M)$ of the $k$-algebra $M$. Its elements are all derivations of $M$. If $M$ is any $k$-algebra and $D,D'$ are two derivations of $M$, then the composite mapping $DD'$ is not a derivation of $M$ in general, but $DD'-D'D$ is. Thus, for any $k$-algebra $M$, we can define the Lie $k$-algebra $\Der_k(M)$ as the subset of $\End(_kM)$ consisting of all derivations of $M$ with multiplication $[D,D']:=DD'-D'D$ for every $D,D'\in\Der_k(M)$.

	It is known that there is not a general notion of representation (or module) over our (non-associative) $k$-algebras. There is a notion of bimodule over a non-associative ring due to Eillenberg, and this notion works well for Lie algebras, but is not convenient in the study of Jordan algebras and alternative algebras. The situation, as far as modules are concerned, is the following.

\subsection{Modules over an associative $k$-algebra.}	Given any $k$-algebra $M$, we can consider, for every element $x\in M$, the mapping $\lambda_x\colon M\to M$, defined by $\lambda_x(a)=xa$ for every $a\in M$. The mapping $\lambda\colon M\rightarrow \mathrm{End}(_kM)$ is defined by $\lambda\colon x\mapsto\lambda_x$ for every $x\in~M$. This $\lambda$ is a $k$-algebra morphism if and only if $M$ is associative. Thus, for any associative $k$-algebra $M$, it is natural to define a {\em left $M$-module} as any $k$-module $_kA$ with a $k$-algebra homomorphism $\lambda\colon M\to\End(_kA)$. Similarly, we can define {\em right $M$-modules} as $k$-modules $_kA$ with a $k$-algebra antihomomorphism $\rho\colon M\to \End(_kA)$. Here by {\em $k$-algebra antihomomorphism} $\psi\colon M\to M'$ between two $k$-algebras $M,M'$ we mean any $k$-linear mapping $\psi$ such that $\psi(xy)=\psi(y)\psi(x)$ for every $x,y\in M$. Clearly, a mapping $ M\to M'$ is a $k$-algebra antihomomorphism if and only if it is a $k$-algebra homomorphism $M^{\op}\to M'$. It follows that right $M$-modules coincide with left $M^{\op}$-modules. More precisely, when we say that right $M$-modules coincide with left $M^{\op}$-modules, we mean that there is a canonical category isomorphism between the category of all right $M$-modules and the category of all left $M^{\op}$-modules. Similarly, left $M$-modules coincide with right $M^{\op}$-modules. Also, if $M$ is commutative, then left $M$-modules and right $M$-modules coincide. Finally, left modules $A$ over an associative $k$-algebra $M$ can be equivalently defined using, instead of the $k$-algebra homomorphism $\lambda\colon M\to\End(_kA)$, a $k$-bilinear mapping $\mu\colon M\times A\to A$, $\mu\colon (m,a)\mapsto ma$, such that $(mm')a=m(m'a)$ for every $m,m'\in M$ and $a\in A$.
		
\subsection{Modules over a Lie $k$-algebra.}	 For any $k$-module $A$ we will denote by $\mathfrak g\mathfrak l(A)$ the Lie $k$-algebra $\End(_kA)$ of all $k$-endomorphisms of $A$ with the operation $[-,-]$ defined by $[f,g]=fg-gf$. 

For any Lie $k$-algebra $M$ and any element $x\in M$, the mapping $\lambda_x$ is an element of the Lie $k$-algebra $\Der_k(M)$, usually called the {\em adjoint} of $x$, or the {\em inner derivation} defined by $x$, and usually denoted by $\ad_Mx$ instead of $\lambda_x$, and the mapping $\ad\colon M\to\Der_k(M)\subseteq \mathfrak g\mathfrak l(M)$, defined by $\ad\colon x\mapsto\ad_Mx$ for every $x\in M$, is a
 Lie $k$-algebra homomorphism. 

 {\em Left modules} over a Lie $k$-algebra $M$ are defined as  $k$-modules $A$ with a Lie $k$-algebra homomorphism 
$\lambda\colon M\to\mathfrak g\mathfrak l(A)$. Similarly, it is possible to define {\em right $M$-modules} as $k$-modules  
$A$ with a $k$-algebra antihomomorphism $\rho\colon M\to \mathfrak g\mathfrak l(A)$. But any Lie $k$-algebra 
$M$ is isomorphic to its opposite algebra $M^{\op}$ via the isomorphism $M\to M^{\op}$, $x\mapsto-x$. 
It follows that the category of right $M$-modules is canonically isomorphic to the category of left 
$M$-modules for any Lie $k$-algebra $M$. Therefore it is useless to introduce both right and left modules, 
it is sufficient to introduce left $M$-modules and call them simply ``$M$-modules''. 
			
	\section{Pre-Lie $k$-algebras }\label{3}
	
			A \emph{pre-Lie $k$-algebra} is a $k$-algebra $(M,\cdot)$ satisfying the identity
		\begin{equation}
			(x\cdot y)\cdot z-x\cdot(y\cdot z)=(y\cdot x)\cdot z-y\cdot(x\cdot z)\label{321}
		\end{equation}
		for every $x,y,z\in M$.
			
					For any $k$-algebra $(M,\cdot)$, defining the commutator $[x,y]=x\cdot y-y\cdot x$ for every $x,y\in M$, the algebra $(M,[-,-])$ is anticommutative. If $(M,\cdot)$ is a pre-Lie algebra, one gets that $(M,[-,-])$ is a Lie algebra, called the {\em Lie algebra sub-adjacent} to the pre-Lie algebra $(M,\cdot)$. 

\medskip

	Pre-Lie algebras are also called Vinberg algebras or left-symmetric algebras. This last name refers to the fact that in (\ref{321}) one exchanges the first two variables on the left. A \emph{right-symmetric algebra} is an algebra in which, for every $x,y,z \in M$, 
		$(x\cdot y)\cdot z-x\cdot(y\cdot z)=(x\cdot z)\cdot y-x\cdot(z\cdot y)$.		It is easily seen that the category of left-symmetric algebras and the category of right-symmetric algebras are isomorphic (the categorical isomorphism is given by $M\mapsto M^{\op}$).

\begin{examples} {\rm  (1) Every associative algebra is clearly a pre-Lie algebra. 

(2) {\em Derivations on $k[x_1,\dots ,x_n]^n$.}
Let $k$ be a commutative ring with identity, $n\geq 1$ be an integer, and $k[x_1,\dots ,x_n]$ be the ring of polynomials in the $n$ indeterminates $x_1,\dots ,x_n$ with coefficients in $k$. Let $A$ be the free $k[x_1,\dots ,x_n]$-module $k[x_1,\dots ,x_n]^n$ with free set $\{e_1,\dots ,e_n\}$ of generators. As a $k$-module, $A$ is the free $k$-module with free set of generators the set $\{\,x_1^{i_1}\dots x_n^{i_n}e_j \mid i_1,\dots ,i_n\geq 0,\  j=1,\dots ,n\}$. Consider the usual derivations of the ring $k[x_1,\dots ,x_n]$: 
$$ \frac{\partial}{\partial x_j}(x_1^{i_1}\dots x_n^{i_n}) = 
\begin{cases}
	x_1^{i_1}\dots i_jx_j^{i_j-1}\dots x_n^{i_n}  & \text{for  $i_j> 0$},\\
	0\ & \text{for $i_j=0$.}
\end{cases} $$
Define a multiplication on $A$ setting, for every $u=(u_1,\dots ,u_n),v=(v_1,\dots ,v_n)\in~A$,
$$ v\cdot u=(\sum_{j=1}^{n}v_j\frac{\partial u_1}{\partial x_j},\dots ,\sum_{j=1}^{n}v_j\frac{\partial u_n}{\partial x_j}). $$ It is then possible to see that $A$ is a pre-Lie $k$-algebra \cite[Section~2.3]{tesi}.

\bigskip

(3) {\em An example of rank $2$.} Let $k$ be any commutative ring with identity and $L\cong k\oplus k$ a free $k$-module of rank $2$ with free set $\{e_1,e_2\}$ of generators. Define a multiplication on $L$ setting $e_1e_1=2e_1,\ e_1e_2=e_2,\ e_2e_1=0,\ e_2e_2=e_1$, and extending by $k$-bilinearity. Then $L$ is a pre-Lie $k$-algebra \cite{NAVW}.

\bigskip

(4) 
{\em Rooted trees.}
Recall that a \emph{tree} is an undirected graph in which any two vertices are connected by exactly one path, or equivalently a connected acyclic undirected graph.
	A \emph{rooted tree} of degree $n$ is a pair $(T,r)$, where $T$ is a tree with $n$ vertices, and its  \emph{root} $r$ is a vertex of $T$. In the following we will label the vertices of $T$ with the numbers $1,\dots ,n$, and the root $r$ with $1$.
	
Let $k$ be a commutative ring with identity and $\mathcal{T}_{n}$ be the free $k$-module with free set of generators the set of all isomorphism classes of rooted trees of degree $n$. Set $$\mathcal{T}:=\bigoplus_{n\geq 1}\mathcal{T}_n.$$

Define a multiplication on $\mathcal{T}$ setting, for every pair $T_1,T_2$ of rooted trees, $$T_1\cdot T_2=\sum_{v\in V(T_2)}T_1\circ_vT_2,$$ where $V(T_2)$ is the set of vertices of $T_2$, and $T_1\circ_vT_2$ is the rooted tree obtained by adding to the disjoint union of $T_1$ and $T_2$ a further new edge joining the root vertex of $T_1$ with the vertex $v$ of $T_2$. The root of $T_1\circ_vT_2$ is defined to be the same as the root of $T_2$. To get a multiplication on $\mathcal{T}$, extend  this multiplication  by $k$-bilinearity.

Let us give an example. Suppose
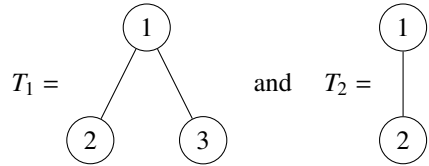
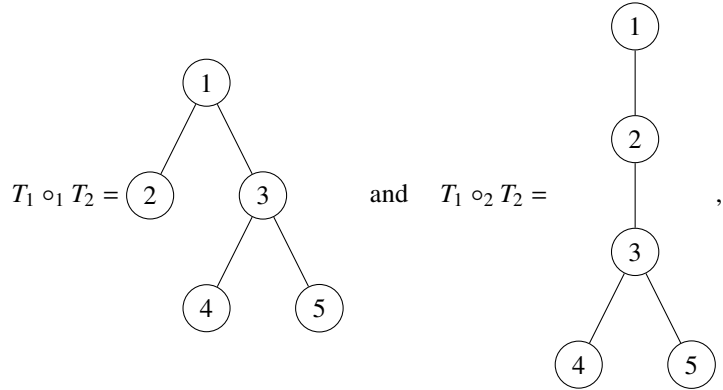
\[T_1=\vcenter{\hbox{\begin{tikzpicture}
			\node[circle,draw] {$1$}
			child {node[circle,draw] {$2$}}
			child {node[circle,draw] {$3$}};
\end{tikzpicture}}} \quad \mbox{and } \quad
T_2=\vcenter{\hbox{\begin{tikzpicture}
			\node[circle,draw] {$1$}
			child {node[circle,draw] {$2$}};
\end{tikzpicture}}}\]
Then \[T_1\circ_1T_2=\vcenter{\hbox{\begin{tikzpicture}
			\node[circle,draw] {$1$}
			child {node[circle,draw] {$2$}}
			child {node[circle,draw] {$3$}
					child {node[circle,draw] {$4$}}
					child {node[circle,draw] {$5$}}
				};
\end{tikzpicture}}} \quad \mbox{and } \quad
T_1\circ_2T_2=\vcenter{\hbox{\begin{tikzpicture}
			\node[circle,draw] {$1$}
			child {node[circle,draw] {$2$}
					child {node[circle,draw] {$3$}
						child {node[circle,draw] {$4$}}
						child {node[circle,draw] {$5$}}
				}
			};
\end{tikzpicture}}},\]
where we have relabelled the vertices of $T_1$. (If $T_1$ has $n$ vertices and $T_2$ has $m$ vertices, it is convenient to relabel in $T_1\circ_vT_2$ the vertices $1,\dots,n$ of $T_1$ with the numbers $m+1,\dots,m+n$, respectively.) Therefore
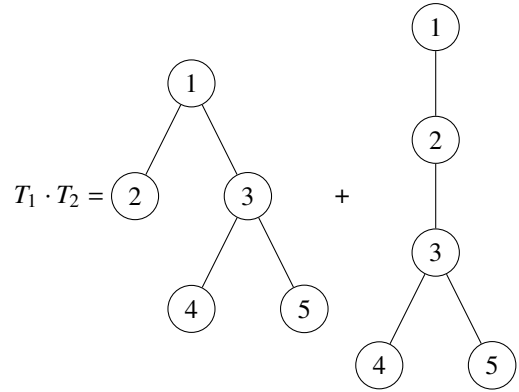
\[T_1\cdot T_2=
\vcenter{\hbox{\begin{tikzpicture}
			\node[circle,draw] {$1$}
			child {node[circle,draw] {$2$}}
			child {node[circle,draw] {$3$}
				child {node[circle,draw] {$4$}}
				child {node[circle,draw] {$5$}}
			};
\end{tikzpicture}}}
+
\vcenter{\hbox{\begin{tikzpicture}
			\node[circle,draw] {$1$}
			child {node[circle,draw] {$2$}
				child {node[circle,draw] {$3$}
					child {node[circle,draw] {$4$}}
					child {node[circle,draw] {$5$}}
				}
			};
\end{tikzpicture}}}\]

In this way, one gets a pre-Lie $k$-algebra $\mathcal{T}$ \cite{CL, tesi}. It is a graded $k$-algebra because $\mathcal{T}_{n}\cdot\mathcal{T}_{m}\subseteq \mathcal{T}_{n+m}$ for every $n$ and $m$. It can be proved that this is the free pre-Lie $k$-algebra on one generator \cite{CL}. (The free generator of $\mathcal{T}$ is the rooted tree with one vertex.)

\medskip

(5) {\em Upper triangular matrices.} This is an interesting example taken from \cite{NAVW}, where all the details can be found. Let $k$ be a commutative ring with identity in which $2$ is invertible, and $n$ be a fixed positive integer. Let $M$ be the $k$-algebra of all $n\times n$ matrices, and $U$ be the its subalgebra of upper triangular matrices. Let $\varphi\colon M\to U$ be the the $k$-linear mapping that associates with any matrix $A=(a_{ij})\in M$ the matrix $B=(b_{ij})\in U$, where $b_{ij}=a_{ij}$ if $a_{ij}$ is above the main diagonal, $b_{ij}=0$ if $a_{ij}$ is below the main diagonal, and $b_{ii}=a_{ii}/2$ if $a_{ij}=a_{ii}$ is on the main diagonal. Also, for every $A\in M$, let $A^{\tr}$ be the transpose of the matrix $A$. Define an operation $\cdot$ on $U$ setting, for every $X,Y\in U$, $X\cdot Y:=XY+\varphi(XY^{\tr}+YX^{\tr})$. Then $(U,\cdot)$ is a pre-Lie $k$-algebra.}
\end{examples}

As we have defined in Section \ref{1}, a $k$-algebra homomorphism $\varphi\colon M \rightarrow M'$ is a $k$-module morphism such that $\varphi(xy)=\varphi(x)\varphi(y)$ for every $x,y\in M$. But we also need another notion. We say that a $k$-module morphism $\varphi\colon M \rightarrow M'$, where $M,M'$ are arbitrary (not-necessarily associative) $k$-algebras, is a {\em pre-morphism} if $\varphi(xy)-\varphi(x)\varphi(y)=\varphi(yx)-\varphi(y)\varphi(x)$ for every $x,y\in M$. 

\begin{lemma}\label{nuovo} A mapping $\varphi\colon M\to M'$, where $(M,\cdot),(M',\cdot)$ are arbitrary $k$-algebras, is a pre-morphism $(M,\cdot)\to (M',\cdot)$ if and only if it is a $k$-algebra morphism $(M,[-,-])\to (M',[-,-])$.\end{lemma}

\begin{proof} If $(M,\cdot)$, $(M',\cdot)$ are $k$-algebras and $\varphi\colon M \rightarrow M'$ is a mapping, then $$\varphi\colon (M,\cdot)\to (M',\cdot)$$ is a pre-morphism if and only if  $\varphi(ab)-\varphi(a)\varphi(b)=\varphi(ba)-\varphi(b)\varphi(a)$ for every $a,b\in M$. This equality can be re-written as $\varphi(ab)-\varphi(ba)=\varphi(a)\varphi(b)-\varphi(b)\varphi(a)$, that is, $\varphi([a,b])=[\varphi(a),\varphi(b)]$. \end{proof}

From this lemma and the definition of pre-morphism, we immediately get that:

\begin{lemma}\label{222} {\rm (a)} Every $k$-algebra morphism is a pre-morphism.

{\rm (b)} The composite mapping of two pre-morphisms is a pre-morphism.

{\rm (c)} The inverse mapping of a bijective pre-morphism is a pre-morphism. \end{lemma}

In Section \ref{1}, we already considered, for  any (not-necessarily  associative) $k$-algebra $M$,  the mapping $\lambda\colon M\rightarrow \End(_kM)$, where $\lambda\colon x\mapsto\lambda_x$, $\lambda_x\colon M\to M$, and $\lambda_x(a)=xa$. Also, we had already remarked that this mapping $\lambda$ is a $k$-algebra morphism if and only if $M$ is associative. The mapping $\lambda$ is a pre-morphism if and only if $M$ is a pre-Lie algebra.

\medskip

There is a category of $k$-algebras with pre-morphisms, i.e., a category in which objects are $k$-algebras and the $\Hom$-set of all morphisms $M\to M'$ consists of all pre-morphisms $M\to M'$. This category contains as a full subcategory the category $\PreL_{k,p}$
of 
pre-Lie $k$-algebras (with pre-morphisms). The category $\PreL_{k,p}$ contains as a subcategory the category $\PreL_{k}$ of pre-Lie algebras with $k$-algebra morphisms, hence a fortiori the category of associative algebras with their morphisms.

\medskip

From lemma \ref{nuovo}, we get 

\begin{theorem}\label{7.1} Associating with any $k$-algebra $(A,\cdot)$ its sub-adjacent  anticommutative algebra $(A,[-,-])$ is a functor $U$ from the category of $k$-algebras with pre-morphisms to the category of anticommutative $k$-algebras.\end{theorem}


Notice that the functor $U$, viewed as a functor from the category $\PreL_{k,p}$ to the category of Lie $k$-algebras, is fully faithful. Two pre-Lie algebras $A,A'$ are isomorphic in $\PreL_{k,p}$ if and only if their sub-adjacent Lie algebras are isomorphic Lie algebras. Two pre-Lie algebras isomorphic in $\PreL_{k,p}$ are not necessarily isomorphic as pre-Lie algebras. The simplest example is, over the field $\R$ of real numbers, the example of the two $\R$-algebras $\R\times\R$ and $\C$. They are non-isomorphic associative commutative $2$-dimensional $\R$-algebras, so that their sub-adjacent Lie algebras are both the $2$-dimensional abelian Lie $\R$-algebra. Hence $\R\times\R$ and $\C$ are isomorphic objects in $\PreL_{\R,p}$. All $\R$-linear mappings $\R\times\R\to\C$ are pre-morphisms.

\begin{remark} More generally, a $k$-algebra $A$ is said to be {\em Lie-admissible} if, setting $[x,y]=xy-yx$, one gets a Lie algebra $(A,[-,-])$. If the {\em associator }of a $k$-algebra~$A$ is defined as
$(x, y, z) = (xy)z -x(yz)$ for all $x,y,z$ in $A$, then being a pre-Lie algebra is equivalent to $(x, y, z) = (y, x, z)$ for all $x,y,z\in A$. Being a Lie-admissible algebra is equivalent to \begin{equation}(x, y, z) + (y, z, x) + (z, x, y) = (y, x, z)  + (x, z, y)+ (z, y, x)\label{Liea}\end{equation} for every $x,y,z\in A$. Pre-Lie algebras are Lie-admissible algebras. By lemma~\ref{nuovo}, the functor $U\colon (A,\cdot)\mapsto (A,[-,-])$ is a fully faithful functor from the category of Lie-admissible $k$-algebras with pre-morphisms to the category of Lie $k$-algebras.\end{remark}

Corresponding to the notion of pre-morphism, there is a notion of pre-derivation. We say that a $k$-module endomorphism $\delta\colon M \rightarrow M$, where $M$ is an arbitrary (not-necessarily associative) $k$-algebra, is a {\em pre-derivation} if $$\delta(xy)-\delta(x)y-x\delta(y)=\delta(yx)-\delta(y)x-y\delta(x)$$ for every $x,y\in M$. 

\begin{lemma} \label{vhip}Let $k$ be a commutative ring with identity, $(A,\cdot)$ a $k$-algebra, and $[-,-]\colon$\linebreak $A\times A\to A$ the operation on $A$ defined by $[x,y]:=xy-yx$ for every $x,y\in A$. Then a $k$-module endomorphism $\delta$ of $A$ is a pre-derivation of $(A,\cdot)$ if and only if it is a derivation of the $k$-algebra $(A,[-,-])$.\end{lemma}

\begin{proof} The $k$-module endomorphism $\delta$ of $A$ is a pre-derivation of $(A,\cdot)$ if and only if $\delta(xy)-\delta(x)y-x\delta(y)=\delta(yx)-\delta(y)x-y\delta(x)$, that is, $\delta([x,y])=[\delta(x), y]+[x,\delta(y)]$. \end{proof}

\begin{proposition} {\rm (a)} Every derivation of a $k$-algebra is a pre-derivation.

{\rm (b)} If $\delta$ and $\delta'$ are two pre-derivations of a $k$-algebra $A$, then $[\delta,\delta']:=\delta\circ\delta'-\delta'\circ\delta$ is a pre-derivation. \end{proposition}

\begin{proof} (a) is trivial,
and
(b) follows from lemma~\ref{vhip}.\end{proof}

\begin{corollary} For any $k$-algebra $A$, the set $\PreDer_k(A)$ of all pre-derivations of $A$ is a Lie $k$-algebra with the operation $[-,-]$ defined by $[\delta,\delta']:=\delta\circ\delta'-\delta'\circ\delta$ for every $\delta,\delta'\in \PreDer_k(A)$.\end{corollary}

\begin{proof} The $k$-algebra $(\PreDer_k(A),[-,-])$ is the Lie algebra of all derivations of the $k$-algebra $(A,[-,-])$ (lemma~\ref{vhip}).\end{proof}

\begin{proposition}\label{2.9} Let $(A,\cdot)$ be any $k$-algebra. For every $x\in A$ define a $k$-module morphism $d_x\colon A\to A$ setting $d_x(y):=xy-yx$ for every $y\in A$. 
The following conditions are equivalent:

{\rm (a)} $d_x$ is a pre-derivation for all $x\in A$, that is, the image $d(A)$ of the mapping $d\colon A\to\End(_kA)$ is contained in $\PreDer_k(A)$.

{\rm (b)} The mapping $d$ is a pre-morphism of the $k$-algebra $(A,\cdot)$ into the associative $k$-algebra $(\End(_kA),\circ)$.

{\rm (c)} The $k$-algebra $(A,\cdot)$ is Lie-admissible.\end{proposition}

\begin{proof} (a)${}\Leftrightarrow{}$(c) The mapping $d_x\colon (A,\cdot)\to (A,\cdot)$ is a pre-derivation if and only if the mapping $d_x\colon (A,[-,-])\to (A,[-,-])$ is a derivation by lemma~\ref{vhip}, i.e., if and only if $d_x([y,z])=[d_x(y),z]+[y,d_x(z)]$. Since the mapping $d_x$ is defined by $d_x(y)=[x,y]$, this is equivalent to  $[x,[y,z]]=[[x,y],z]+[y,[x,z]]$, for every $x,y,z\in A$. This proves that $d_x$ is a pre-derivation for every $x\in A$ if and only if $(A,[-,-])$ is a Lie algebra, that is, if and only if $(A,\cdot)$ is Lie-admissible.

(b)${}\Leftrightarrow{}$(c)  The mapping $d$ is a pre-morphism if and only if $d_{xy}-d_x\circ d_y=d_{yx}-d_y\circ d_x$ for every $x,y\in A$, that is, if and only if $d_{xy}(z)-d_x( d_y(z))=d_{yx}(z)-d_y(d_x(z))$ for every $x,y,z\in A$. This is equivalent to $(xy)z-z(xy)-d_x(yz-zy)=(yx)z-z(yx)-d_y(xz-zx)$. An easy calculation shows that this is exactly Condition~(\ref{Liea}), i.e., it is equivalent to the fact that $A$ is Lie-admissible.\end{proof}

If $A$ is a Lie-admissible $k$-algebra, the mapping $d_x$ is the {\em inner pre-derivation} of $A$ induced by $x$.

\section{Pre-Lie algebras are modules over the sub-adjacent  Lie algebra}\label{xxx}

Now we want to give another presentation of pre-Lie algebras, helpful to understand their structure. 

Let $k$ be a commutative ring with identity. Given a pre-Lie $k$-algebra 	$(A,\cdot)$, we have already seen in the paragraph after Lemma~\ref{222}
that the mapping $\lambda\colon (A,\cdot)\to \End(_kA)$ is a pre-morphism. Apply to it the functor $U$, getting a Lie $k$-algebra morphism $L:=U(\lambda)\colon (A,[-,-])\rightarrow\mathfrak{gl}(A)$ defined by $L\colon a\mapsto \lambda_{a}$ for every $a\in A$. This mapping $L$ is set-theoretically equal to the mapping $\lambda$. In other words, $L$ defines a module structure on the $k$-module $_kA$, giving it the structure of a module over the sub-adjacent  Lie $k$-algebra $(A,[-,-])$. Moreover, $[x,y]=L(x)(y)-L(y)(x)$.

This construction can be inverted. Let $(A,[-,-])$ be a Lie $k$-algebra, and suppose that its sub-adjacent  $k$-module $_kA$ has a module structure over the Lie algebra $(A,[-,-])$  via the Lie algebra morphism $L\colon (A,[-,-])\rightarrow\mathfrak{gl}(A)$ and that, for every $x,y\in A$, the condition $L(x)(y)-L(y)(x)=[x,y]$ holds. Define a new multiplication $\cdot$ on  $A$ setting $x\cdot y=L(x)(y)$ for every $x,y\in A$. Then $(A,\cdot)$ turns out to be a pre-Lie $k$-algebra. These two constructions are one the inverse of the other. More precisely, fix a Lie $k$-algebra $A$. Then there is a category isomorphism between the following two categories $\Cal S_A$ and $\Cal M_A$, where:

(1) $\Cal S_A$ is the category whose objects are all pre-Lie $k$-algebras $(A,\cdot)$ whose sub-adjacent  Lie algebra is the fixed Lie algebra $(A,[-,-])$. The morphisms are all pre-Lie algebra homomorphisms between such pre-Lie algebras.

(2) $\Cal M_A$  is the category whose objects are all pre-Lie $k$-algebra morphisms $L\colon (A,[-,-])\rightarrow\mathfrak{gl}(A)$ such that $L(x)(y)-L(y)(x)=[x,y]$ for every $x,y\in A$. The morphisms $\varphi:L\rightarrow L'$ between two objects $L,L'$ of $\Cal M_A$  are the $k$-module morphisms $\varphi\colon A\to A$ for which all diagrams 
\[\begin{tikzcd}
	M \arrow[r, "\varphi"] \arrow[d, "L(a)"'] & M \arrow[d, "L'(\varphi(a))"] \\
	M \arrow[r, "\varphi"]                    & M                            
\end{tikzcd} \]
commute, for every $a\in A$. See \cite[Theorem~1.2.7]{tesi}.

\bigskip

\subsection{Modules over a pre-Lie $k$-algebra.}\label{yyy}	 Modules cannot be defined over arbitrary non-associative algebras, but the definition of pre-Lie algebra immediately suggests us how it is possible to define modules over a pre-Lie algebra.

A {\em module} $M$ over a pre-Lie $k$-algebra $A$ is any $k$-module $M$ with a $k$-bilinear mapping $\cdot\colon A\times M\to M$ such that
		\begin{equation}
			(x\cdot y)\cdot m-x\cdot(y\cdot m)=(y\cdot x)\cdot m-y\cdot(x\cdot m)\label{123}
		\end{equation}
		for every $x,y\in A$ and $m\in M$.
		
		\medskip
		
		Like in the case of associative algebras, it is possible to equivalently define a module $M$ over a pre-Lie $k$-algebra $(A,\cdot)$ as any $k$-module $M$ with a pre-morphism $\lambda\colon (A,\cdot)\to (\End(_kM),\circ)$.
		
		\medskip

		For instance, if $A$ is any pre-Lie $k$-algebra and $I$ is an ideal of $I$, taking as $k$-bilinear mapping $\cdot\colon A\times I\to I$ the restriction of the multiplication on $A$, one sees immediately that $I$ is a module over $A$.
		
\begin{theorem}\label{same modules} The category of modules over a pre-Lie $k$-algebra $(A,\cdot)$ and the category of modules over its sub-adjacent  Lie $k$-algebra $(A,[-,-])$ are isomorphic.\end{theorem}

\begin{proof} 
Modules over the pre-Lie algebra $(A,\cdot)$ are pairs $(_kM,\lambda)$ with $_kM$ a $k$-module and $\lambda\colon A\to\End(_kM)$ a pre-morphism, and modules over the Lie algebra $(A,[-,-])$ are pairs $(_kM,\lambda)$ with $_kM$ a $k$-module and $\lambda\colon (A,[-,-])\to\mathfrak{gl}(M)$ a Lie $k$-algebra morphism. By Lemma~\ref{nuovo}, they are the same pairs.
\end{proof}

Notice that we could have obtained the results in Section \ref{xxx} in a different way: every pre-Lie algebra is clearly a module over itself, hence, applying Theorem~\ref{same modules}, to every pre-Lie algebra $(A,\cdot)$ there corresponds a module $A_k$ over the sub-adjacent  Lie algebra $(A,[-,-])$, that is, a Lie algebra morphism $L\colon (A,[-,-])\to \mathfrak{gl}(A)$, and $[x,y]=L(x)(y)-L(y)(x)$ for every $x,y\in A$.

Also notice that the modules we have defined in this section over a pre-Lie algebra are left modules. We don't consider right modules because the definition of pre-Lie algebra is not right/left symmetric, that is, the opposite of a pre-Lie algebra is not a pre-Lie algebra.

\section{Commutator of two ideals. (Huq=Smith) for pre-Lie algebras}\label{4}

The sum of two ideals of a pre-Lie $k$-algebra $A$, i.e., their sum as $k$-submodules of $A$, is an ideal of $A$, and any  intersection of a family of ideals of $A$ is an ideal of $A$. It follows that the set $\Cal I(A)$ of all ideals of a pre-Lie algebra $A$ is a complete lattice with respect to $\subseteq$, and it is a sublattice of the lattice of all $k$-submodules of $A_k$, hence $\Cal I(A)$ is a modular lattice. Moreover, the ideal of $A$ generated by a subset $X$ of $A$ is the intersection of all the ideals of $A$ that contain $X$.

We now need a notion of commutator of two ideals of a pre-Lie algebra. The variety $\mathcal{V}$ of pre-Lie $k$-algebras is a Barr-exact category, is a variety of $\Omega$-groups, is protomodular and is semi-abelian \cite[Example~(2)]{JMT}. More precisely, pre-Lie algebras have an underlying group structure with respect to their addition, so that they have the Mal'tsev term $p(x,y,z)=x-y+z$. See \cite[Proposition~5.3.1]{BB}. Notice that $p(p(x,y,0),x,y))=0$ for every $x,y\in A$, hence the variety $\mathcal{V}$ of pre-Lie algebras is protomodular by  \cite[Proposition~3.1.8]{BB}.  Moreover, $p$ has the property that $p(p(x,y,t),t,z)=p(x,y,z)$ for all $x,y,z,t\in A$ {\em (semi-associativity)}, so $\mathcal{V}$ is semi-abelian by  \cite[Proposition~5.3.3]{BB}.

We want to show that  the Huq and the Smith commutators of two ideals of a pre-Lie $k$-algebra coincide. Recall that in the case of the semi-abelian variety $\mathcal{V}$ of pre-Lie algebras, the Huq commutator of two ideals $I$ and $J$ of a pre-Lie algebra $A$ is the smallest ideal $[I,J]_H$ of $A$ for which there is a well-defined canonical morphism $I\times J\to A/[I,J]_H$ such that $(i,0)\mapsto i+[I,J]_H$ and $(0,j)\mapsto j+[I,J]_H$ for every $i\in I$ and $j\in J$. That is, $[I,J]_H$  is the smallest ideal of $A$ for which the mapping $I\times J\to A/[I,J]_H$, defined by $(i,j)\mapsto i+j+[I,J]_H$  for every $i\in I$ and $j\in J$, is a pre-Lie algebra morphism. 

\begin{proposition} The Huq commutator $[I,J]_H$ of two ideals $I$ and $J$ of a pre-Lie algebra $A$ is the ideal of $A$ generated by the subset $\{\, ij, ji\mid i\in I, \ j\in J\,\}$.\end{proposition}

\begin{proof}The mapping $\bar{\sigma}\colon I\times J\rightarrow A/[I,J]_H$, defined by $(i,j)\mapsto i+j+[I,J]_H$, is a pre-Lie $k$-algebra morphism if and only if it respects multiplication, that is, if and only if $\bar{\sigma}((i,j)\cdot(i',j'))\equiv\bar{\sigma}(i,j)\bar{\sigma}(i',j')$ for every $(i,j),(i',j')\in I\times J$, that is, if and only if $ii'+jj'\equiv(i+j)(i'+j')$ modulo $[I,J]_H$. Hence $\bar{\sigma}$ is a pre-Lie algebra morphism if and only if $ij'+ji'\equiv0$ modulo $[I,J]_H$, i.e., if and only if $ij'+ji'\in[I,J]_H$. The conclusion follows immediately.\end{proof}

The {\em Smith commutator}  in the Mal'tsev variety $\mathcal{V}$ (see \cite{15}) can be defined, for a pre-Lie $k$-algebra $A$ with
Mal'tsev term $p(x, y, z)$ and two ideals $I,J$ of $A$, as the smallest ideal $[I,J]_S$ of $A$ for which the function
$$p \colon\{(x, y, z) \mid x\equiv y \pmod I,\ y\equiv z \pmod J\} \to A/[I,J]_S,$$ defined by $p
 (x, y, z) =x-y+z+[I,J]_S$ is a pre-Lie algebra morphism.

\begin{theorem} The Smith commutator $[I,J]_S$ of two ideals $I$ and $J$ of a pre-Lie algebra $A$ is the ideal of $A$ generated by the subset $\{\, ij, ji\mid i\in I, \ j\in J\,\}$. Hence Huq=Smith for pre-Lie algebras. \end{theorem}

\begin{proof}
	The mapping $p\colon\{\,(b+i,b,b+j)\mid b\in A,\ i\in I,\ j\in J\,\}\rightarrow A/[I,J]_S$ is a pre-Lie algebra morphism if and only if for every $b,b'\in A$, $i,i'\in I$, $j,j'\in J$, one has $$p((b+i,b,b+j)(b'+i',b',b'+j'))\equiv p(b+i,b,b+j)p(b'+i',b',b'+j') \pmod{[I,J]_S},$$
	that is, $p((b+i)(b'+i'),bb',(b+j)(b'+j'))\equiv(b+i+j)(b'+i'+j')\;\mathrm{mod}[I,J]_S$. Equivalently, if and only if $0\equiv ij'+ji'\;\mathrm{mod}[I,J]_S$. Therefore the Smith commutator $[I,J]_S$ of the two ideals $I$ and $J$  is the ideal of $A$ generated by the subset $\{\, ij, ji\mid i\in I, \ j\in J\,\}$. In particolar, $[I,J]_H=[I,J]_S$.
\end{proof}

From now on we will not distinguish between the Huq commutator $[I,J]_H$ and the Smith commutator $[I,J]_S$. We will simply call it the {\em commutator} of the two ideals $I$ and $J$. Notice that the commutator is commutative, in the sense that $[I,J]=[J,I]$. 

Let us briefly discuss the structure of this ideal $[I,J]$. It is clear that if $X$ is any subset of a pre-Lie $k$-algebra $A$, the ideal $\langle X\rangle$ of $A$ generated by $X$, that is, the intersection of all the ideals of $A$ that contain $X$, can be also described as the union $\langle X\rangle=\bigcup_{n\ge 0} X_n$ of the following ascending chain $X_0\subseteq X_1\subseteq\dots$ of $k$-submodules of $A$: $X_0$ is the $k$-submodule of $A$ generated by $X$; given $X_n$, set $X_{n+1}=X_n+AX_n+X_nA$, where $AX_n$ denotes the set of all finite sums of products $ax$ with $a\in A$ and $x\in X_n$, and similarly for $X_nA$. In the case of the ideal $[I,J]$ this specializes as follows:

\begin{proposition}\label{commutator} Let $I$ and $J$ be ideals of a pre-Lie $k$-algebra $A$. Then $$[I,J]=IJ+\sum_{n\ge 0}S_n,$$ where $S_n=((\dots(((JI)A)A)\dots)A)A$ and in $S_n$ there are $n$ factors equal to $A$ on the right of the factor J$I$.\end{proposition}

\begin{proof} {\em Step 1: $A(IJ)\subseteq IJ$.}

By Property~(\ref{321}), we have that $A(IJ)\subseteq (AI)J+(IA)J+I(AJ)\subseteq IJ$.

\medskip

 {\em Step 2: $A(JI)\subseteq JI$.}
 
 From Step 1, by symmetry.
 
 \medskip
 
  {\em Step 3: $AS_n\subseteq S_n+S_{n+1}$ for every $n\ge0$.}
  
  Induction on $n$. Step 2 gives the case $n=0$. Suppose that $AS_n\subseteq S_n+S_{n+1}$ for some $n\ge0$. Then $AS_{n+1}=A(S_nA)\subseteq (AS_n)A+(S_nA)A+S_n(AA)\subseteq (S_n+S_{n+1})A+S_{n+2}+S_{n+1}=S_{n+1}+S_{n+2}$.

\medskip
 
  {\em Step 4: $S_nA=S_{n+1}$.}
  
  By definition.
  
  \medskip
 
  {\em Step 5: $(IJ)A\subseteq IJ+S_0+S_1$.}
  
  In fact, $(IJ)A\subseteq I(JA)+(JI)A+J(IA)\subseteq IJ+S_1+S_0$.
  
  \medskip
 
  {\em Final Step.}
  
  Clearly, $IJ+\sum_{n\ge 0}S_n$ is a $k$-submodule of $A$ that contains $IJ$ and $JI$ and is contained in the ideal generated by $IJ\cup JI$. Hence it remains to show that it is closed by left and right multiplication by elements of $A$. This is proved in Steps 1, 3, 4 and 5.
 \end{proof}

Now that we have a good notion of commutator of two ideals $I$ and $J$ of a pre-Lie $k$-algebra $A$, we can introduce the multiplicative lattice of all ideals of $A$: it is the complete modular lattice $\Cal I(A)$ of all ideals of $A$ endowed with the commutator of ideals. Notice that, trivially, $[I,J]\subseteq I\cap J$. As a consequence of looking at pre-Lie algebras from the point of view of multiplicative lattices, we immediately get the notions of prime ideal of a pre-Lie $k$-algebra $A$, (Zariski) prime spectrum of $A$, semiprime ideal, abelian pre-Lie algebra, idempotent (=perfect) pre-Lie algebra, derived series, solvable pre-Lie algebra, lower central series, nilpotent pre-Lie algebra, $m$-system, $n$-system, hyperabelian pre-Lie algebra, metabelian pre-Lie algebra, Jacobson radical, centralizer of an ideal, center of a pre-Lie $k$-algebra, hypercenter. See the next Section~\ref{5} and \cite{FFJ, Francesco, F, primes-mult}.

\medskip

Notice that {\em the monotonicity condition holds} for our commutator of ideals of a pre-Lie algebra $A$, in the sense that if $I\le I'$ and $J\le J'$ are ideals of $A$, then $[I,J]\le [I',J']$.

\medskip

Also notice that the description of the commutator in Proposition~\ref{commutator} reduces, in the case of $I=J=A$, to the equality $[A,A]=A^2=AA$. Here $A^2$ is the image of the $k$-module morphism $\mu\colon A\otimes_kA\to A$ induced by the $k$-bilinear multiplication of $A$.

\section{The commutator is not associative}\label{5}

In this section we will show that the commutator of ideals in a pre-Lie algebra $A$ is not associative in general, that is, if $I,J,K$ are ideals of $A$, it is not necessarily true that $[I,[J,K]]=[[I,J],K]$. In our example, the algebra $A$ will be factor algebra $A:=\mathcal{T}/P$, where $\mathcal{T}$ is the pre-Lie algebra of rooted trees of Example~4 in Section~\ref{3}, and $P$ is the ideal of $\mathcal{T}$ generated by all rooted trees with at least $5$ vertices. Such $P$ is the $k$-submodule of $\mathcal{T}$ generated by all rooted trees with at least $5$ vertices. The rooted trees with at most $4$ vertices up to isomorphism are \[v=\vcenter{\hbox{\begin{tikzpicture}
			\node[circle,draw] {$1$};
\end{tikzpicture}}}\quad , \quad 
e=\vcenter{\hbox{\begin{tikzpicture}
			\node[circle,draw] {$1$}
			child {node[circle,draw] {$2$}};
\end{tikzpicture}}}, \quad a=\vcenter{\hbox{\begin{tikzpicture}
			\node[circle,draw] {$1$}
			child {node[circle,draw] {$2$}}
			child {node[circle,draw] {$3$}};
\end{tikzpicture}}},
            \]
			\[ \quad b= \vcenter{\hbox{\begin{tikzpicture}
			\node[circle,draw] {$1$}
			child {node[circle,draw] {$2$}
					child {node[circle,draw] {$3$}
				}
			};\end{tikzpicture}}}\quad , \quad c=\vcenter{\hbox{\begin{tikzpicture}
			\node[circle,draw] {$1$}
			child {node[circle,draw] {$2$}}
			child {node[circle,draw] {$3$}}
			child {node[circle,draw] {$4$}};
\end{tikzpicture}}}, \]
			\[
			 d=
\vcenter{\hbox{\begin{tikzpicture}
			\node[circle,draw] {$1$}
			child {node[circle,draw] {$2$}
				child {node[circle,draw] {$3$}}
				child {node[circle,draw] {$4$}}
			};
\end{tikzpicture}}}, \quad f=\vcenter{\hbox{\begin{tikzpicture}
			\node[circle,draw] {$1$}
			child {node[circle,draw] {$2$}}
			child {node[circle,draw] {$3$}
					child {node[circle,draw] {$4$}}
				};
\end{tikzpicture}}}\quad , \quad g=
\vcenter{\hbox{\begin{tikzpicture}
			\node[circle,draw] {$1$}
			child {node[circle,draw] {$2$}
				child {node[circle,draw] {$3$}
					child {node[circle,draw] {$4$}}
				}
			};
\end{tikzpicture}}}.\]
Hence our pre-Lie $k$-algebra $A$ is eight dimensional, and we will denote by $v,e,a,b,$\linebreak $c,d,f,g$ the images in $A$ of the corresponding rooted trees. That is, we will say that $\{v,e,a,b,c,d,f,g\}$ is a free set of generators for the free $k$-module $A$. From the multiplication in $\mathcal{T}$ defined in Example~4 of Section~\ref{3}, we get that the multiplication table in $A$ is  \bigskip

  \begin{center}
    \begin{tabular}{|c|c|c|c|c|c|c|c|c|}
    \hline
          & $v$ & $e$ & $a$ & $b$ & $c$ & $d$ & $f$ & $g$  \\
  \hline
        $v$  & $e$ & $a+b$ & $c+2f$ & $f+d+g$ & $0$ & $0$ & $0$ & $0$  \\ \hline
        $e$  & $b$ & $f+g$ & $0$ & $0$ & $0$ & $0$ & $0$ & $0$  \\ \hline
         $a$ & $d$ & $0$ & $0$ & $0$ & $0$ & $0$ & $0$ & $0$  \\  \hline
      $b$    & $g$ & $0$ & $0$ & $0$ & $0$ & $0$ & $0$ & $0$  \\  \hline
      $c$    & $0$ & $0$ & $0$ & $0$ & $0$ & $0$ & $0$ & $0$  \\ \hline
       $d$   & $0$ & $0$ & $0$ & $0$ & $0$ & $0$ & $0$ & $0$  \\  \hline
        $f$  &  $0$ & $0$ & $0$ & $0$ & $0$ & $0$ & $0$ & $0$  \\  \hline
      $g$    & $0$ & $0$ & $0$ & $0$ & $0$ & $0$ & $0$ & $0$  \\
       \hline      
    \end{tabular}
  \end{center}
   
    \bigskip
    
    \noindent From the multiplication table we see that $A^2=AA$ has $\{e,b,d,g,a,f,c\}$ as a set of generators, and is a seven  dimensional free $k$-module.
    
    Now $[A^2,A^2]=\sum_{n\ge 0}(\dots((A^2\cdot A^2)\cdot A)\cdot\cdots\cdot A)\cdot A$, where there are $n$ factors equal to $A$ on the right. But, always from the multiplication table, one sees that $A^2\cdot A^2$ is generated by $f+g$. Moreover $(f+g)A=0$ and $A(f+g)=0$. Therefore $[A^2,A^2]$ is one dimension as a free $k$-module, and its free set of generators is $\{f+g\}$.
    
    Similarly, $[A^2,A]=A\cdot A^2+\sum_{n\ge 1}(\dots((A^2\cdot A)\cdot A)\cdot\cdots\cdot A)\cdot A$, where there are $n+1$ factors equal to $A$ on the right. From the multiplication table, we see that $A\cdot A^2$ is generated by $a+b,f+g,c+2f,f+d+g$. Also, $A^2\cdot A$ is generated by $\{b,d,g,f+g\}$, $(A^2\cdot A)\cdot A$ is generated by $g$, and $((A^2\cdot A)\cdot A)\cdot A=0$. Therefore $[A^2,A]$ is the $k$-module generated by $b,d,g,f,a,c$ and is six dimensional. It follows that $[A^2,A]\cdot A$ is generated by $\{d,g\}$, $A\cdot([A^2,A])$ is generated by $\{c+2f,f+d+g\}$, and $([A^2,A]\cdot A)\cdot A=0$. From these equalities we get that $[[A^2,A], A]$ is generated by $\{d,g,c+2f,f+d+g\}$. Equivalently, $[[A^2,A], A]$ is generated by $\{d,g,f,c\}$ and is four dimensional. In particular $[A^2,A^2]\ne [[A^2,A], A]$.

\medskip

Let's illustrate in detail some of the notions that immediately derive from the commutative multiplication $[-,-]$ (the commutator) in the multiplicative lattice $\Cal I(A)$.

\medskip

First of all, a pre-Lie $k$-algebra $A$ is {\em abelian} if the commutator of $A$ and itself is zero: $[A,A]=0$. This is equivalent to saying that $ij=0$ for every $i,j\in A$. That is, a pre-Lie algebra $(A,\cdot)$ is abelian if and only if $x\cdot y=0$ for every $x,y\in A$.
(This is equivalent to requiring that the addition $+\colon A\times A\to A$ is a pre-Lie algebra morphism.)

\medskip

By definition, an ideal $I$ of a pre-Lie $k$-algebra $A$ is {\em prime} if it is properly contained in $A$ and, for every ideal $J,K$ of $A$, $[J,K]\subseteq I$ implies $J\subseteq I$ or $K\subseteq I$. An ideal $I$ of a pre-Lie $k$-algebra $A$ is {\em semiprime} if, for every ideal $J$ of $A$, $[J,J]\subseteq I$ implies that $J\subseteq I$. An ideal of $A$ is semiprime if and only if it is the intersection of a family of prime ideals (if and only if it is the intersection of all the ideals of $A$ that contain it). An ideal $P$ of a pre-Lie $k$-algebra $A$ is prime if and only if the lattice $\Cal I(A/P)$ is  uniform and $A/P$ has no non-zero abelian ideal.

\begin{remark} Instead of the commutator $[I,J]$ of two ideals $I$ and $J$, we could have taken two other ``product of ideals'' in a pre-Lie $k$-algebra: we could consider the product $IJ$, i.e., the $k$-submodule of $A$ generated by all products $ij$, which is a $k$-submodule but not an ideal of $A$ in general, or the ideal $\langle IJ\rangle$ generated by the submodule $IJ$. Notice that $IJ\subseteq \langle IJ\rangle\subseteq [I,J]=\langle IJ\rangle+\langle JI\rangle$, where the last equality follows from Proposition~\ref{commutator}. Correspondingly, we would have had three different notions of ``prime ideal''. In the next proposition (essentially contained in \cite[Example~3.7]{FFJ}) we prove that these three notions of ``prime ideal'' coincide:

\begin{proposition} \label{associ} The following conditions are equivalent for an ideal $P$ of a pre-Lie algebra $A$:

{\rm (a)} If $I,J$ are ideals of $A$ and $IJ\subseteq P$, then either $I\subseteq P$ or $J\subseteq P$.

{\rm (b)} If $I,J$ are ideals of $A$ and $\langle IJ\rangle \subseteq P$, then either $I\subseteq P$ or $J\subseteq P$.

{\rm (c)} If $I,J$ are ideals of $A$ and $[I,J]\subseteq P$, then either $I\subseteq P$ or $J\subseteq P$.\end{proposition}

\begin{proof} The implications (a)${}\Rightarrow{}$(b)${}\Rightarrow {}$(c) follow immediately from the fact that $IJ\subseteq \langle IJ\rangle\subseteq [I,J]$.

(c)${}\Rightarrow {}$(a). Let $P$ satisfy condition (c) and fix two ideals $I,J$ of $A$ such that $IJ\subseteq P$. Since $P$ is an ideal, it follows that $\langle IJ\rangle\subseteq P$. Also, $[\langle JI\rangle, \langle JI\rangle]=\langle\langle JI\rangle\langle JI\rangle\rangle\le \langle IJ\rangle\le P$. From (c), we get that $\langle JI\rangle\le P$, so that $[I,J]=\langle IJ\rangle+\langle JI\rangle\le P$. From (c) again, we get that either $I\subseteq P$ or $J\subseteq P$.
\end{proof}

Proposition~\ref{associ} shows that if the pre-Lie algebra $A$ is an associative algebra, then this notion of prime ideal coincide with the notion of prime ideal in an associative algebra. Proposition~\ref{commutator} shows that, for every pair $(I,J)$ of ideals of a pre-Lie algebra $A$, one has $[I,J]=IJ+\langle JI\rangle =JI+\langle IJ\rangle$. Also, Step 5 in the proof of that Proposition shows that one always has that $IJ+JI+(IJ)A=IJ+JI+(JI)A$.\end{remark}

A pre-Lie $k$-algebra $A$ is {\em idempotent} (or {\em perfect}) if $[A,A]=A$, that is, if $A^2=A$ (last paragraph of Section~\ref{4}).

\medskip

Given any pre-Lie algebra $A$, let $\Spec(A)$ be the set of all its prime ideals. For every $I\in\Cal I(A)$, set $V(I)=\{\,P\in\Spec(A)\mid P\supseteq I\,\}$. Then the family of all subsets $V(I)$ of $\Spec(A)$, $I\in\Cal I(A)$, is the family of  all the closed sets for a topology on $\Spec(A)$. With this topology, the topological space $\Spec(A)$ is the {\em (Zariski) prime spectrum} of $A$, and is a sober space \cite {FFJ}. It is not a spectral space in the sense of Hochster in general. For instance, if $B$ is a Boolean ring without identity, then $B$ is a pre-Lie algebra, but its prime spectrum is not compact.  

\medskip

If the pre-Lie algebra $A$ is an associative algebra, then this notion of prime spectrum coincide with the ``standard notion'' of prime spectrum of an associative algebra $A$, where the points of the spectrum are the prime ideals of $A$ and the closed sets are the subsets $V(I)$ of the spectrum. To tell the truth, there is not a ``standard notion'' of prime spectrum of an associative algebra that extends the classical notion of prime spectrum for commutative associative algebras with identity. There are several such notions as it is shown in \cite{Reyes2} and \cite{Reyes}. For instance, the points of the spectrum could be the completely prime ideals of $A$, or the spectrum of $A$ could be defined to be the Zariski spectrum of the commutative ring $A/[A,A]$, where $[A,A]$ now denotes the ideal of $A$ generated by all elements $ab-ba$.

\medskip

A pre-Lie $k$-algebra $A$ is {\em hyperabelian} if it has no prime ideal. For instance, abelian pre-Lie algebras are hyperabelian.

\medskip

Let $A$ be a pre-Lie $k$-algebra. 
The {\em lower central series} (or {\em descending central series}) of $A$ is the descending series $$A=A_1\ge A_2\ge A_3\ge\dots,$$ where $A_{n+1}:=[A_n, A]$ for every $n\ge 1$. If $A_n=0$ for some $n\ge 1$, then $A$ is {\em nilpotent}. (Notice that it is not necessary to distinguish between left nilpotency and right nilpotency, because the commutator is commutative, that is, [$A_n, A]=[A,A_n]$.)

The {\em derived series} of $A$ \cite[Definition~6.1]{FFJ}  is the descending series $$A:=A^{(0)}\ge A^{(1)}\ge A^{(2)}\ge\dots,$$ where $A^{(n+1)}:=[A^{(n)},A^{(n)}]$ for every $n\ge 0$.  The pre-Lie algebra $A$ is {\em solvable} if $A^{(n)}=0$ for some integer $n\ge0$. It is {\em metabelian} if $A^{(2)}=0$.

In a multiplicative lattice an element is {\em semisimple} if it is the join of a set of minimal idempotent elements. (An element $m$ of a lattice $L$ is {\em minimal} if, for every $x\in L$, $x\le m$ implies $x=m$ or $x=0$, that is, if it is minimal in the partially ordered set $L\setminus\{0\}$. An element $e$ of a multiplicative lattice $L$ is {\em idempotent} if $e\cdot e=e$). ``Minimal idempotent element'' of $L$ means minimal element of $L\setminus\{0\}$ that is also an idempotent element. Notice that for a minimal element $x\in L$ either $x\cdot x=x$ or $x\cdot x=0$, i.e., minimal elements are either idempotent or abelian.

The {\em Jacobson radical} of $L$ is the meet of the set of all maximal elements $a$ of $L\setminus\{1\}$ with $1\cdot 1\not\le a$. The {\em radical} is the join of the set of all solvable elements of $L$.

\section{Idempotent endomorphisms, semidirect products of pre-Lie algebras, and actions}

Let $e$ be an idempotent endomorphism of a pre-Lie $k$-algebra $A$. Then $A=\ker(e)\oplus e(A)$ (direct sum as $k$-modules), where the kernel $\ker(e)$ of $e$ is an ideal of $A$ and its image $e(A)$ is a pre-Lie sub-$k$-algebra of $A$. If there is a direct-sum decomposition $A=I\oplus B$ as $k$-module of a pre-Lie $k$-algebra $A$, where $I$ is an ideal of $A$ and $B$ is a pre-Lie sub-$k$-algebra of $A$, we will say that $A$ is the {\em semidirect product} of $I$ and $B$. We are interested in semidirect products because, for any algebraic structure, idempotent endomorphisms are in one-to-one correspondence with semidirect products and are related to the notion of action of the structure on another structure, and bimodules.

\medskip

The proof of the following proposition is elementary.

\begin{proposition} \label{0.1} Let $A$ be  a pre-Lie $k$-algebra, $I$ an ideal of $A$ and $B$ a pre-Lie sub-$k$-algebra of $A$. The following conditions are equivalent:

{\rm $(1)$} $A=I\oplus B$ as a $k$-module.

{\rm $(2)$}    For every $a\in A$, there are a unique $i\in I$ and a unique $b\in B$ such that $a=i+b$.
 
{\rm  $(3)$}    There exists a pre-Lie $k$-algebra morphism $A\to B$ whose restriction to $B$ is the identity and whose kernel is $I$.  
 
{\rm  $(4)$ } There is an idempotent pre-Lie $k$-algebra endomorphism of $A$ whose image is $B$ and whose kernel is $I$.\end{proposition}
 
It is now clear that there is a one-to-one correspondence between the set of all idempotent endomorphisms of a pre-Lie $k$-algebra  $A$ and the set of all pairs $(I,B)$, where $I$ is an ideal of $A$, $B$ is a pre-Lie sub-$k$-algebra of $A$, and $A$ is the direct sum of $I$ and $B$ as a $k$-module.

\bigskip

Let us first consider inner semidirect product. Suppose  that $(A,\cdot)$ is a pre-Lie $k$-algebra that is a semidirect product of its ideal $I$ and its pre-Lie sub-$k$-algebra $B$. Then there is a pre-morphism $\lambda\colon (B,\cdot)\to (\End(I_k),\circ)$, given by multiplying on the left by elements of $B$ (this follows from the fact that every ideal is a module, as we have already remarked in Section~\ref{yyy}). Also, there is a $k$-module morphism $\rho\colon B\to \End(I_k)$, given by multiplying on the right by elements of $B$, that is, $\rho\colon b\mapsto\rho_b$, where $\rho_b(i)=i\cdot b$ for every $i\in I$. Moreover, Identity (\ref{321}), applied to elements $x,z$ in $B$ and $y\in I$, can be re-written as $\rho_a(\lambda_b(i))-\lambda_b(\rho_a(i))=(\rho_a\circ\rho_b-\rho_{b\cdot a})(i)$ for every $a,b\in B$ and $i\in I$. Identity (\ref{321}), applied to elements $x$ in $B$ and $y,z\in I$, can be re-written as $\lambda_a(i)\cdot j-\lambda_a(i\cdot j)=\rho_a(i)\cdot j-i\cdot\lambda_a(j)$ for every $a\in B$ and $i,j\in I$. Finally, the same identity (\ref{321}), applied to elements $z$ in $B$ and $x,y\in I$, can be re-written as $\rho_a(x\cdot y)-x\cdot\rho_a(y)=\rho_a(y\cdot x)-y\cdot\rho_a(x) $ for every $a\in B$ and $i,j\in I$.

Conversely, for outer semidirect product:

\begin{theorem}\label{vilu} Let $I$ and $B$ be pre-Lie $k$-algebras and $(\lambda,\rho)$ a pair of $k$-linear mappings $B\to\End(I_k)$ such that:

{\rm (a)} $\lambda\colon (B,\cdot)\to (\End(I_k),\circ)$ is a pre-morphism.

{\rm (b)} $\rho_a\circ\lambda_b-\lambda_b\circ\rho_a=\rho_a\circ\rho_b-\rho_{b\cdot a}$ for every $a,b\in B$.

{\rm (c)} $\lambda_a(i)\cdot j-\lambda_a(i\cdot j)=\rho_a(i)\cdot j-i\cdot\lambda_a(j)$ for every $a\in B$ and $i,j\in I$.

{\rm (d)} $\rho_a(i\cdot j)-i\cdot\rho_a(j)=\rho_a(j\cdot i)-j\cdot\rho_a(i) $ for every $a\in B$ and $i,j\in I$.
\\ On the $k$-module direct sum $I\oplus B$ define a multiplication $*$ setting $$(i,b)*(j,c)=(i\cdot j+\lambda_b(j)+\rho_c(i), b\cdot c)$$ for every $(i,b),(j,c)\in I\oplus B$. Then $(I\oplus B, *)$ is a pre-Lie $k$-algebra.\end{theorem}

\begin{proof} For every $a,b,c\in B$ and $x,y,z\in I$ we have that
\begin{equation}\begin{array}{l} ((x,a)*(y,b))*(z,c)=(x\cdot y+\lambda_a(y)+\rho_b(x), a\cdot b)*(z,c)=\\
=((x\cdot y)\cdot z+\lambda_a(y)\cdot z+\rho_b(x)\cdot z+\lambda_{a\cdot b}(z)+\\
+\rho_c(x\cdot y+\lambda_a(y)+\rho_b(x)),(a\cdot b)\cdot c)\end{array}\label{a}\end{equation}
and 
\begin{equation}\begin{array}{l} (x,a)*((y,b)*(z,c))=(x,a)*(y\cdot z+\lambda_b(z)+\rho_c(y), b\cdot c)=\\
=(x\cdot(y\cdot z)+x\cdot\lambda_b(z)+x\cdot\rho_c(y)+\\
+\lambda_a(y\cdot z+\lambda_b(z)+\rho_c(y))+\rho_{b\cdot c}(x), a\cdot(b\cdot c)).\end{array}\label{b}\end{equation}
The difference of (\ref{a}) and (\ref{b}) is 
$$\begin{array}{l} ((x\cdot y)\cdot z-x\cdot(y\cdot z)+\lambda_a(y)\cdot z-\lambda_a(y\cdot z)+\\
\qquad +\rho_b(x)\cdot z-x\cdot\lambda_b(z)+
\lambda_{a\cdot b}(z)-(\lambda_a\circ\lambda_b)(z)+  \\
\qquad+\rho_c(x\cdot y)-x\cdot\rho_c(y)
+\rho_c(\lambda_a(y))-\lambda_a(\rho_c(y))+\rho_c(\rho_b(x))-\rho_{b\cdot c}(x)), \\
\qquad(a\cdot b)\cdot c-a\cdot(b\cdot c)).\end{array}$$
Similarly,
$$\begin{array}{l} ((y,b)*(x,a))*(z,c)- (y,b)*((x,a)*(z,c))= \\ \qquad=((y\cdot x)\cdot z-y\cdot(x\cdot z)+\lambda_b(x)\cdot z-\lambda_b(x\cdot z)+\rho_a(y)\cdot z-y\cdot\lambda_a(z)+\\
\qquad+
\lambda_{b\cdot a}(z)-(\lambda_b\circ\lambda_a)(z)+  \rho_c(y\cdot x)-y\cdot\rho_c(x)
+\rho_c(\lambda_b(x))-\lambda_b(\rho_c(x))+\\
\qquad+\rho_c(\rho_a(y))-\rho_{a\cdot c}(y)), (b\cdot a)\cdot c-b\cdot(a\cdot c)).\end{array}$$
Hence, for the proof, it suffices to show that
\begin{equation}\begin{array}{l} \lambda_a(y)\cdot z-\lambda_a(y\cdot z)+\rho_b(x)\cdot z-x\cdot\lambda_b(z)+
\lambda_{a\cdot b}(z)-(\lambda_a\circ\lambda_b)(z)+  \\
+\rho_c(x\cdot y)-x\cdot\rho_c(y)
+\rho_c(\lambda_a(y))-\lambda_a(\rho_c(y))+\rho_c(\rho_b(x))-\rho_{b\cdot c}(x))= \\ =\lambda_b(x)\cdot z-\lambda_b(x\cdot z)+\rho_a(y)\cdot z-y\cdot\lambda_a(z)+\\
+
\lambda_{b\cdot a}(z)-(\lambda_b\circ\lambda_a)(z)+\\
+  \rho_c(y\cdot x)-y\cdot\rho_c(x)
+\rho_c(\lambda_b(x))-\lambda_b(\rho_c(x))+\\
+\rho_c(\rho_a(y))-\rho_{a\cdot c}(y)).\end{array}\label{c}\end{equation}
Now $$\begin{array}{ll}   \lambda_a(y)\cdot z-\lambda_a(y\cdot z)=\rho_a(y)\cdot z-y\cdot\lambda_a(z) & \mbox{\rm by hypotheses (c);} \\
\rho_b(x)\cdot z-x\cdot\lambda_b(z)=\lambda_b(x)\cdot z-\lambda_b(x\cdot z)  & \mbox{\rm by hypotheses (c);} \\
\lambda_{a\cdot b}(z)-(\lambda_a\circ\lambda_b)(z)=\lambda_{b\cdot a}(z)-(\lambda_b\circ\lambda_a)(z)  & \mbox{\rm by hypotheses (a);} \\
\rho_c(x\cdot y)-x\cdot\rho_c(y)=\rho_c(y\cdot x)-y\cdot\rho_c(x)  & \mbox{\rm by hypotheses (d);} \\
\rho_c(\lambda_a(y))-\lambda_a(\rho_c(y))=\rho_c(\rho_a(y))-\rho_{a\cdot c}(y)) & \mbox{\rm by hypotheses (b);} \\
\rho_c(\rho_b(x))-\rho_{b\cdot c}(x))=\rho_c(\lambda_b(x))-\lambda_b(\rho_c(x)) & \mbox{\rm by hypotheses (b).}\end{array}$$
Summing up these equalities one gets Equality (\ref{c}).
\end{proof} 

Hence the theorem characterises the four properties that an action $(\lambda,\rho)$, that is, a pair of $k$-linear mappings $B\to\End(I_k)$, must have in order to construct the semidirect product of a pre-Lie $k$-algebra $B$ acting on a pre-Lie $k$-algebra $I$. 

\bigskip

\subsection{Bimodules over a pre-Lie algebra} The most important case of semidirect product is probably when the pre-Lie algebra $I$ is abelian, i.e., the case where the action, that is, the pair $(\lambda,\rho)$ of $k$-linear mappings $B\to\End(I_k)$, is an action of the pre-Lie $k$-algebra $B$ on a $k$-module $M$. In other words, when $I$ is a $B$-bimodule. Let us be more precise, giving the precise definition of what a bimodule over a pre-Lie algebra must be:

\begin{definition} Let $A$ be a pre-Lie $k$-algebra. A {\em bimodule} over $A$ is a $k$-module $M_k$ with a pair $(\lambda,\rho)$ of $k$-linear mappings $A\to\End(M_k)$ such that:

{\rm (a)} $\lambda\colon (A,\cdot)\to (\End(M_k),\circ)$ is a pre-morphism (that is, $M$ is a module over $A$).

{\rm (b)} $\rho_a\circ\lambda_b-\lambda_b\circ\rho_a=\rho_a\circ\rho_b-\rho_{b\cdot a}$ for every $a,b\in B$.\end{definition}

Notice that Conditions (c) and (d) of Theorem~\ref{vilu} are always trivially satisfied because in this case the $k$-module $M$ is viewed as an abelian pre-Lie algebra, that is, with null multiplication.
This definition already appears, for instance, in \cite{NAVW}. Notice the nice interpretation of condition (b) given in that paper: In condition (b) the left hand side $\rho_a\circ\lambda_b-\lambda_b\circ\rho_a$ describes how far the action is from associativity  (for bimodules over an associative algebra, it is always required to be zero); the right hand side $\rho_a\circ\rho_b-\rho_{b\cdot a}$ describes how far $\rho$ is from being a $k$-algebra antihomomorphism.

\subsection{Adjoining the identity to a pre-Lie algebra} The class of pre-Lie algebras contains the class of associative algebras. For associative algebras, it is very natural to consider associative algebras with an identity, and when there is not an identity, to adjoin one. This construction is often called the ``Dorroh extension''. Let's show that this is possible for pre-Lie algebras as well. We will see in fact that a more appropriate name for our class of algebras, instead of ``pre-Lie algebras'', would have been ``pre-associative algebras''. Adjoining an identity to a pre-Lie $k$-algebra $A$ is exactly our semidirect product of the pre-Lie $k$-algebra $k$ acting on the pre-Lie $k$-algebra $A$. Let's be more precise.

An {\em identity} in a pre-Lie $k$-algebra $A$ is an element, which we will denote by $1_A$, such that $a\cdot1_A=1_A\cdot a=a$ for every $a\in A$. If $A$ has an identity, we will say that $A$ is {\em unital}. An element $e$ of $A$ is {\em idempotent} if $e^2:=e \cdot e=e$. The zero of $A$ is always an idempotent element of $A$, and the identity, when it exists, is also an idempotent element of $A$.
	
	Let $A$ be any fixed pre-Lie $k$-algebra. Then the associative commutative ring $k$ is a pre-Lie $k$-algebra, and there is a one-to-one correspondence between the set of all the pre-Lie $k$-algebra morphisms $k\to A$ and the set of all idempotent elements of $A$. For any idempotent element $e$ of $A$ the corresponding morphism $\varphi_e\colon k\to A$ is defined by $\varphi_e(\lambda)=\lambda e$ for every $\lambda \in k$. Conversely, for any morphism $\varphi\colon k\to A$ the corresponding idempotent element of $A$ is $\varphi(1)$.

For any fixed pre-Lie $k$-algebra $A$ it is possible to construct the semidirect product of $k$ acting on $A$ via the pair $(\lambda,\rho)$ of $k$-module morphisms $k\to\End(A_k)$ for which $\lambda_\alpha=\rho_\alpha$ is multiplication by $\alpha$ for all $\alpha\in k$. Then the four conditions (a), (b), (c), (d) of Theorem~\ref{vilu} are all automatically satisfied, and the corresponding semidirect product is the $k$-module direct sum $A\oplus k$ with the multiplication defined by $$(x,\alpha)(y,\beta)=(x\cdot y+\beta x+\alpha y, \alpha\beta)$$ for every $(x,\alpha),(y,\beta)\in A\oplus k$. Hence $A\oplus k$ becomes a pre-Lie $k$-algebra with identity $(0,1)$. The Lie algebra sub-adjacent  this pre-Lie algebra $A\oplus k$ is the direct sum of the Lie algebra $(A,[-,-])$ and the abelian Lie algebra $k$. We will denote this semidirect product by $A\# k$.

\medskip

Now let  $\PreL_{k,1}$ be the category of all unital pre-Lie $k$-algebras. Its objects are the pre-Lie $k$-algebras $A$ with an identity. Its morphisms $f\colon A\to B$ are the $k$-algebra morphisms $f$ such that $f(1_A)=1_B$. There is also a further category involved. It is the category $\PreL_{k,1,a}$ of all unital pre-Lie $k$-algebras with an {\em augmentation}. Its objects are all the pairs $(A,\varepsilon_A)$, where $A$ is a unital pre-Lie $k$-algebra and $\varepsilon_A\colon A\to k$ is a morphism in $\PreL_{k,1}$ that is a left inverse for $\varphi_{1_A}$: $$\xymatrix@1{
k\ar[r]^{\varphi_{1_A}} &A\ar[r]^{\varepsilon_A} & k. }$$ The morphisms $f\colon (A,\varepsilon_A)\to (B,\varepsilon_{B})$ are the morphisms $f\colon A\to B$ in $\PreL_{k,1}$ such that $\varepsilon_{B}f=\varepsilon_A$. For instance, the $k$-algebra $A\# k$ is clearly a unital $k$-algebra with augmentation: the augmentation  is the canonical projection $\pi_2\colon A\# k=A\oplus k\to k$ onto the second summand.
	
	\bigskip
	
	It is easy to see that:
	
	\begin{theorem}\label{164} There is a category equivalence $F\colon \PreL_{k}\to \PreL_{k,1,a}$ that associates with any object $A$ of $\PreL_{k}$ the $k$-algebra with augmentation $F(A):=(A\# k, \pi_2)$. The quasi-inverse of $F$ is the functor $\PreL_{k,1,a}\to \PreL_{k,}$ that associates with each unital pre-Lie $k$-algebra with augmentation $(A,\varepsilon_A)$ the kernel $\ker(\varepsilon_A)$ of the augmentation. \end{theorem}

\input{references}

\end{document}

%% file: references.tex
%
%
%
\biblstarthook{}